\documentclass[a4paper,11pt]{amsart}
\usepackage{amsmath}
\usepackage{amsthm}
\usepackage{amsfonts}
\usepackage{amssymb}
\usepackage{latexsym}
\usepackage{mathrsfs}

\usepackage{enumitem}

\usepackage[section]{placeins}
\usepackage{xcolor}

\usepackage[abs]{overpic}
\usepackage[titletoc]{appendix}
\usepackage{graphicx}
\usepackage{latexsym}
\usepackage[utf8]{inputenc}
\usepackage{epsfig}
\usepackage{psfrag}

\usepackage{tikz}
\usepackage[normalem]{ulem}

\usepackage{diagbox}

\numberwithin{equation}{section}

\newtheorem{theorem}{Theorem}[section]{\bf}{\it}
\newtheorem{lemma}[theorem]{Lemma}{\bf}{\it}
\newtheorem{proposition}[theorem]{Proposition}{\bf}{\it}
\newtheorem{corollary}[theorem]{Corollary}{\bf}{\it}
{\bf}{\it} 
{\bf}{\it}
\newtheorem*{theorem*}{Theorem}

\newtheorem{remark}[theorem]{Remark}
{\bf}{\it}
{\bf}{\it}

\newtheorem*{namedtheorem}{\theoremname}
\newcommand{\theoremname}{testing}
\newenvironment{named}[1]{\renewcommand{\theoremname}{#1}\begin{namedtheorem}}{\end{namedtheorem}}

\newtheorem*{definition*}{Definition}

{\bf}{\it}
{\bf}{\it}
{\bf}{\it}
\newtheorem{example}[theorem]{Example}
\newtheorem*{example*}{Example}

\theoremstyle{remark}

\theoremstyle{definition}

\theoremstyle{remark}

\newcommand{\R}{\mathbb R}

\newcommand{\C}{\mathbb C}
\newcommand{\N}{\mathbb N}

\newcommand{\loc}{{\operatorname{loc}}}

\newcommand{\id}{{\operatorname{id}}}

\newdimen\vintkern\vintkern11pt
\def\vint{-\kern-\vintkern\int}

\newcommand{\norm}[1]{\lVert #1 \rVert}

\newcommand{\bS}{\mathbb{S}}

\newcommand{\cF}{\mathcal{F}}

\newcommand{\vol}{\mathrm{vol}}

\newcommand{\dR}{\mathrm{dR}}

\newcommand{\cW}{\mathscr W}

\newcommand{\bSpan}{\mathrm{span}}

\newcommand{\im}{\mathrm{im} \,}

\newcommand{\sym}{\mathrm{sym}}

\newcommand{\core}{\mathrm{core} \,}

\newcommand{\abs}[1]{\left\lvert #1 \right\rvert}

\newcommand{\K}{K^n(N)}
\newcommand{\F}{\cF_{K,D}(N,\omega)}
\newcommand{\weakto}{\rightharpoonup}

\renewcommand{\le}{\leqslant}
\renewcommand{\ge}{\geqslant}

\title{Quasiregular curves and cohomology}

\author{Susanna Heikkilä}
\address{Department of Mathematics and Statistics, P.O. Box 68 (Pietari Kalmin katu 5), FI-00014 University of Helsinki, Finland}
\email{susanna.a.heikkila@helsinki.fi}

\thanks{This work was supported in part by the Academy of Finland project \#332671.}

\begin{document}

\begin{abstract}
Let $N$ be a closed, connected, and oriented Riemannian manifold, which admits a quasiregular $\omega$-curve $\R^n \to N$ with infinite energy. We prove that, if the de Rham class of $\omega$ is non-zero and belongs to a so-called Künneth ideal, then there exists a non-trivial graded algebra homomorphism $H_\dR^*(N) \to \bigwedge^* \R^n$ from the de Rham algebra $H_\dR^*(N)$ of $N$ to the exterior algebra $\bigwedge^* \R^n$. As an application, we give examples of pairs $(N,\omega)$, where $N$ is a closed manifold and $\omega$ is a closed $n$-form for $n<\dim N$, for which every quasiregular $\omega$-curve $\R^n \to N$ is constant.
\end{abstract}

\maketitle

\section{Introduction}

A continuous map $f\colon M\to N$ between oriented Riemannian $n$-manifolds, $n\ge 2$, is \emph{$K$-quasiregular for $K\ge 1$} if $f\in W_\loc^{1,n}(M,N)$ and
\[
\norm{Df}^n \le K(\star f^* \vol_N) \text{ a.e.~in } M;
\]
here $\norm{Df}$ is the operator norm of the weak differential $Df$ and $\star f^* \vol_N$ is the Hodge star dual of the $n$-form $f^* \vol_N$. An oriented and connected $n$-manifold $N$ is \emph{quasiregularly elliptic}, if there exists a non-constant quasiregular map $\R^n \to N$.

In 1981, Gromov \cite[p. 200]{GromovM:Hypmga} posed a question, whether all closed simply connected manifolds are quasiregularly elliptic. In 2019, Prywes \cite[Corollary 1.2]{PR} answered Gromov's question by giving examples of closed simply connected $4$-manifolds which are not quasiregularly elliptic. The examples were obtained as an application of Prywes' cohomology dimension bound for closed quasiregularly elliptic manifolds \cite[Theorem 1.1]{PR}: \emph{If $N$ is a closed quasiregularly elliptic manifold, then $\dim H_\dR^k(N) \le \dim \bigwedge^k \R^n = \binom{n}{k}$ for each $k$.} In \cite{Heikkila-Pankka}, Pankka and the author showed that, in fact, $H_\dR^*(N)$ is Euclidean in the following sense: \emph{If $N$ is a closed quasiregularly elliptic manifold, then there exists an embedding of graded algebras $H_\dR^*(N)\to \bigwedge^* \R^n$.} This embedding yields a topological classification of closed simply connected quasiregularly elliptic $4$-manifolds \cite[Corollary 1.2]{Heikkila-Pankka}.

In this paper, we prove a generalization of this cohomology embedding theorem for closed manifolds admitting quasiregular curves.

For the definition of a quasiregular curve, let $M$ be an oriented Riemannian $n$-manifold with $n\ge 2$. Let $N$ be an oriented Riemannian manifold and let $\omega \in \Omega^n(N)$ be a closed non-vanishing form, i.e., $d\omega=0$ and $\omega(x) \ne 0$ for every $x\in N$. A continuous map $F\colon M\to N$ is a \emph{$K$-quasiregular $\omega$-curve for $K\ge 1$} if $F\in W_\loc^{1,n}(M,N)$ and
\[
(\norm{\omega} \circ F) \norm{DF}^n \le K(\star F^* \omega) \text{ a.e.~in } M;
\]
here the \emph{pointwise comass norm of $\omega$} is given by
\[
\norm{\omega(x)} = \max \{\omega_x(v_1,\ldots,v_n) \colon v_1,\ldots,v_n \text{ unit vectors in } T_xN\},
\]
$\norm{DF}$ is the operator norm of the weak differential $DF$, and $\star F^* \omega$ is the Hodge star dual of the $n$-form $F^* \omega$.

Quasiregular curves are introduced in \cite{Pankka-QRC} and include, for example, quasi\-regular maps, holomorphic curves, and pseudoholomorphic curves.

Quasiregular curves lead to the following two notions of quasiregular ellipticity: Let $N$ be an oriented and connected Riemannian manifold and let $\omega \in \Omega^n(N)$ be a closed non-vanishing form with $n\ge 2$. We say that the pair $(N,\omega)$ is \emph{quasiregularly elliptic} if there exists a non-constant quasiregular $\omega$-curve $F\colon \R^n \to N$. We also say that the pair $(N,\omega)$ is \emph{infinite energy quasiregularly elliptic} if there exists a quasiregular $\omega$-curve $F\colon \R^n \to N$ satisfying $\int_{\R^n} F^* \omega = \infty$.

\begin{remark}
We note that a pair $(N,\vol_N)$, where $N$ is a closed $n$-manifold satisfying $H_\dR^*(N)\ne H_\dR^*(\bS^n)$, is infinite energy quasiregularly elliptic if and only if $N$ is quasiregularly elliptic. The equivalence follows from the Bonk--Heinonen growth result for quasiregular mappings \cite[Theorem~1.11]{BonkM:Quamc}: \emph{Let $N$ be a closed, connected, and oriented Riemannian $n$-manifold with $n\ge 2$ and $H_\dR^*(N)\ne H_\dR^*(\bS^n)$. Then every non-constant $K$-quasiregular map $f\colon \R^n \to N$ satisfies
\[
\liminf_{r\to \infty} \frac{\int_{B^n(0,r)} f^* \vol_N}{r^\varepsilon} >0,
\]
where $\varepsilon=\varepsilon(n,K)>0$.} To our knowledge, a similar growth result is not known for quasiregular curves in full generality; for so-called \emph{signed quasiregular curves}, see \cite{Heikkila}. In particular, the growth theorem in \cite{Heikkila} guarantees infinite energy in all the examples presented in this paper; for details see Theorem~\ref{thm:signed}.
\end{remark}

To state our main result, we introduce some terminology stemming from the Künneth theorem: \emph{Let $N=M\times P$ be a product of closed smooth manifolds. Then the map
\[
\bigoplus_{\ell=0}^k H_\dR^\ell(M) \otimes H_\dR^{k-\ell}(P) \to H_\dR^k(N), \; c \otimes c' \mapsto \pi_M^* c \wedge \pi_P^* c',
\]
where $\pi_M \colon N\to M$ and $\pi_P \colon N\to P$ are the projections, is an isomorphism for each $k$.}

For a smooth manifold $N$, we say that the \emph{Künneth ideal of $N$} is the ideal $K^*(N) = \oplus_{k=2}^{\dim N} K^k(N)$ of $H_\dR^*(N)$, where the $k$th layer $K^k(N)$ is the vector space generated by products $c\wedge c'$ of de Rham classes $c\in H_\dR^\ell(N)$ and $c'\in H_{\dR}^{k-\ell}(N)$ with $1\le \ell \le k-1$, i.e.,
\[
K^k(N) = \bSpan_\R \left( \bigcup_{\ell=1}^{k-1} H_\dR^\ell(N) \wedge H_\dR^{k-\ell}(N) \right).
\]

We are now ready to present our main result.

\begin{theorem} \label{thm:weaker-main}
Let $(N,\omega)$ be infinite energy quasiregularly elliptic, where $N$ is closed and $0\ne [\omega]_\dR \in \K$. Then there exists a graded algebra homomorphism $\Phi \colon H_{\dR}^*(N) \to \bigwedge^* \R^n$ for which $\Phi [\omega]_\dR \ne 0$.
\end{theorem}

By the Poincaré duality and the Bonk--Heinonen growth theorem, Theorem~\ref{thm:weaker-main} yields as a corollary the aforementioned embedding theorem in \cite{Heikkila-Pankka}.

Before discussing further consequences of Theorem~\ref{thm:weaker-main}, we briefly mention examples of non-trivial Künneth ideals, which are used later in this paper. The first example is given by product manifolds, where the factors are equidimensional. 

\begin{example} \label{ex:product-kunneth}
Let $n\ge 2$. Let $N=N_1 \times \cdots \times N_m$ and $\omega=\sum_{i=1}^m \pi_i^* \vol_{N_i} \in \Omega^n(N)$, where each $N_i$ is a closed, connected, and oriented Riemannian $n$-manifold with $H_\dR^*(N_i)\ne H_\dR^*(\bS^n)$ and each $\pi_i \colon N\to N_i$ is the projection. Then, by the Künneth theorem and Poincaré duality, $0\ne [\omega]_\dR \in \K$. In particular, if $m=1$, then $0\ne [\vol_N] \in \K$.
\end{example}

Kähler manifolds also give examples of non-trivial Künneth ideals.

\begin{example} \label{ex:symplectic-kunneth}
Let $(N,\omega_\sym)$ be a Kähler manifold. Let $m$ be an integer satisfying $4\le 2m\le \dim N$ and let $\omega = \omega_\sym^{\wedge m} = \omega_\sym \wedge \cdots \wedge \omega_\sym \in \Omega^{2m}(N)$. Then $0\ne [\omega]_\dR \in K^{2m}(N)$.
\end{example}

\subsection*{Examples of non-ellipticity}

We obtain from Theorem~\ref{thm:weaker-main} two families of examples of closed manifolds $N$ and closed non-vanishing forms $\omega$ for which the pair $(N,\omega)$ is not quasiregularly elliptic. Our first family of examples consists of pairs $(N,\omega)$, where $N$ is a $6$-manifold and $\omega$ is a $4$-form.

\begin{corollary} \label{cor:(non-)existence-QRC-type1}
Let $g\ge 1$ and $T_g = \#^g (\bS^1 \times \bS^1)$. Let $\omega_\sym \in \Omega^2(\C P^2)$ be the standard symplectic form. Then $(T_g \times \C P^2,\pi_1^* \vol_{T_g} \wedge \pi_2^* \omega_\sym)$ is quasiregularly elliptic if and only if $g=1$.
\end{corollary}

Corollary~\ref{cor:(non-)existence-QRC-type1} also shows that the quasiregular ellipticity of a pair $(N,\omega)$ depends on both $N$ and $\omega$. Indeed, $(T_g \times \C P^2,\pi_2^* \omega_\sym)$ is quasiregularly elliptic for every $g\ge 1$.

Our second family of not quasiregularly elliptic examples are pairs $(N,\omega)$, where $N$ is a simply connected $8$-manifold and $\omega$ is a $6$-form.

\begin{corollary} \label{cor:(non-)existence-QRC-type2}
Let $\nu \ge 8$ and $S_\nu = \#^\nu (\bS^2 \times \bS^2)$. Let $\omega_\sym \in \Omega^2(\C P^2)$ be the standard symplectic form. Then $(S_\nu \times \C P^2,\pi_1^* \vol_{S_\nu} \wedge \pi_2^* \omega_\sym)$ is not quasiregularly elliptic.
\end{corollary}

It is not clear to us, whether the lower bound $\nu \ge 8$ is sharp in Corollary~\ref{cor:(non-)existence-QRC-type2}.

\subsection*{Prywes' theorem for quasiregular curves}

As a byproduct of the proof of Theorem~\ref{thm:weaker-main}, we also obtain the following result, which can be interpreted as a relative of Prywes' theorem for quasiregular curves.

\begin{theorem} \label{thm:submfd}
Let $(N,\omega)$ be infinite energy quasiregularly elliptic, where $N$ is a closed manifold, $\{0\}\ne \bSpan_\R ([\omega]_\dR)=\K$, and $\norm{\omega}>0$. Let $M\subset N$ be a closed, connected, and oriented $n$-submanifold satisfying $\iota^*(H_\dR^*(N))=H_\dR^*(M)$ and $\iota^* [\omega]_\dR \ne 0$, where $\iota \colon M\to N$ is the inclusion. Then $\dim H_\dR^k(M) \le \binom{n}{k}$ for $k=0,\ldots,n$.
\end{theorem}

Theorem~\ref{thm:submfd} has a topological analogue: \emph{Let $F\colon \Sigma \to N$ be a continuous map between smooth manifolds for which $F^*(\K) \ne \{0\}$ for some $n\ge 1$. Let $M\subset N$ be a closed, connected, and oriented $n$-submanifold satisfying $\iota^*(H_\dR^*(N))=H_\dR^*(M)$ and $\ker \iota^* \cap \, \K \subset \ker F^*$, where $\iota \colon M\to N$ is the inclusion. Then there exists an embedding $H_\dR^k(M)\to H_\dR^k(\Sigma)$ for each $k$.} For a formulation of Theorem~\ref{thm:submfd} more similar to this topological analogue, see Theorem~\ref{thm:submfd-with-core}.

\subsection*{Quasiregular curves with respect to closed non-exact forms} Instead of directly proving Theorem~\ref{thm:weaker-main}, we prove a stronger statement for pairs $(N,\omega)$, where $\omega$ is a closed non-exact form, possibly vanishing at some points on $N$.

\begin{theorem} \label{thm:main}
Let $N$ be a closed, connected, and oriented Riemannian manifold. Let $2\le n\le \dim N$ and let $\omega \in \Omega^n(N)$ be a closed form with $0\ne [\omega]_{\dR} \in \K$. Let $F\colon \R^n \to N$ be a quasiregular $\omega$-curve satisfying $\int_{\R^n} F^* \omega = \infty$ and $\inf_{\R^n} (\norm{\omega} \circ F)>0$. Then there exists a graded algebra homomorphism $\Phi \colon H_{\dR}^*(N) \to \bigwedge^* \R^n$ for which $\Phi [\omega]_\dR \ne 0$.
\end{theorem}

\begin{remark}
Formally, $F$ is a quasiregular $\omega_+$-curve $\R^n \to N_+$, where
\[
N_+ = \{ x\in N \colon \omega(x) \ne 0\} \quad \text{ and } \quad \omega_+ = \omega|_{N_+}.
\]
\end{remark}

In Theorem~\ref{thm:main}, we assume that the form is not exact since, if $N$ is a closed, connected, and oriented Riemannian manifold and $\omega \in \Omega^n(N)$, $2\le n\le \dim N$, is an exact form, then every quasiregular $\omega$-curve $F\colon \R^n \to N$ satisfying $\inf_{\R^n} (\norm{\omega} \circ F)>0$ is constant; see \cite[Theorem~1.7]{Pankka-QRC} and \cite[Remark~3.2]{Pankka-QRC}.

\subsection*{Idea of the proof}
The proof of Theorem~\ref{thm:main} follows the proof of the cohomology embedding theorem for closed quasiregularly elliptic manifolds in \cite{Heikkila-Pankka}.

The proof consists of two key ideas. We reduce the global problem of Theorem~\ref{thm:main} to a local setting. For the statements, let $N$ be an oriented Riemannian manifold and let $\omega \in \Omega^n(N)$ be a closed non-exact form, where $2\le n\le \dim N$. For constants $K,D\ge 1$, let $\F$ be the family of non-constant $K$-quasiregular $\omega$-curves $F\colon B_2^n \to N$ satisfying the additional conditions
\[
\int_{B_2^n} F^* \omega \le D \int_{B^n} F^*\omega
\]
and $\inf_{B_2^n} (\norm{\omega} \circ F) \ge D^{-1}$. Here and in what follows, $B^n$ and $B_2^n$ denote the Euclidean balls of radius $1$ and $2$ centered at the origin, respectively. In what follows, we also denote
\[
A_\omega(F) = \int_{B^n} F^* \omega >0
\]
for $F\in \F$.

We are now ready to state a localized reformulation of Theorem~\ref{thm:main}; see also \cite[Theorem~1.3]{Heikkila-Pankka}.

\begin{theorem} \label{thm:sup}
Let $N$ be a closed, connected, and oriented Riemannian manifold. Let $2\le n\le \dim N$ and let $\omega \in \Omega^n(N)$ be a closed form with $0\ne [\omega]_{\dR} \in \K$. Assume also that there exists $K,D\ge 1$ satisfying
\[
\sup_{F\in \F} A_\omega(F) = \infty.
\]
Then there exists a graded algebra homomorphism $\Phi \colon H_{\dR}^*(N) \to \bigwedge^* \R^n$ for which $\Phi [\omega]_{\dR}\ne 0$.
\end{theorem}

The reduction of Theorem~\ref{thm:main} to Theorem~\ref{thm:sup} follows from Rickman's Hunting Lemma. After the reduction, we construct the homomorphism in Theorem~\ref{thm:sup} as the evaluation of a weak limit of normalized pull-backs.

Let $N$ be a closed, connected, and oriented Riemannian manifold and let
\[
h\colon H_\dR^*(N) \to \Omega^*(N), \quad c\mapsto h_c,
\]
be the map associating to each de Rham class $c\in H_\dR^*(N)$ its unique harmonic representative $h_c \in c$; to see that the map $h$ is well-defined see e.g.~\cite[Theorem 6.11]{WA}. Given $\omega$, $K$ and $D$, we define, for each $F\in \F$, a \emph{normalized pull-back} $F^\# \colon H_\dR^*(N) \to \cW^*(B_2^n)$ by formula
\[
F^\#c = A_\omega(F)^{-\frac{k}{n}} F^*(h(c))
\]
for $c\in H_\dR^k(N)$ and $k=0,\ldots,n$. Here and in what follows, $\cW^*(B_2^n)$ is the graded algebra
\[
\cW^k(B_2^n) = \begin{cases}
W^{d,\frac{n}{k}}(B_2^n; {\bigwedge}^k \R^n), &\, \text{ for } 1\le k\le n, \\
\ker(d\colon \Omega^0(B_2^n) \to \Omega^1(B_2^n)), &\, \text{ for } k=0,
\end{cases}
\]
of Sobolev forms on $B_2^n$; for more details see Section~\ref{sec:prelim-sobolev}.

The normalized pull-backs have the following weak limit property: if $[\omega]_{\dR} \in \K$ and $(F_j)$ is a sequence in $\F$ with the properties that $A_\omega(F_j)\to \infty$ and that the sequence
\[
\left(F_j^\#|_{H_\dR^k(N)} \colon H_\dR^k(N) \to \cW^k(B_2^n)\right)
\]
has a weak limit for each $k=1,\ldots,n-1$, then the limit operator extends to a graded algebra homomorphism $H_\dR^*(N) \to \cW^*(B_2^n)$ which maps $[\omega]_\dR$ non-trivially. More precisely, we prove the following theorem; see also \cite[Theorem~1.7]{Heikkila-Pankka}.

\begin{theorem} \label{thm:limit-extension}
Let $N$ be a closed, connected, and oriented Riemannian manifold. Let $2\le n\le \dim N$ and let $\omega \in \Omega^n(N)$ be a closed form with $0\ne [\omega]_{\dR} \in \K$. Let $K,D\ge 1$. Let $(F_j)$ be a sequence in $\F$ for which $A_\omega(F_j)\to \infty$ and $F_j^\# \weakto L$, where
\[
F_j^\#,L \colon \bigoplus_{k=1}^{n-1} H_\dR^k(N) \to \bigoplus_{k=1}^{n-1} L^\frac{n}{k}(B_2^n; {\bigwedge}^k \R^n).
\]
Then $L$ extends to a graded algebra homomorphism $L\colon H_\dR^*(N)\to \cW^*(B_2^n)$ for which $\int_{B^n} L[\omega]_\dR = 1$.
\end{theorem}

In general, the choice of harmonic representative $h\colon H_\dR^*(N) \to \Omega^*(N)$ is not an algebra homomorphism and hence neither are the normalized pull-backs $F_j^\#$. However, Theorem~\ref{thm:limit-extension} shows that the limiting operator $L$ commutes with the exterior product.

\subsection*{Organization of the article}

In Section~\ref{sec:reduction}, we discuss how Theorem~\ref{thm:main} reduces to Theorem~\ref{thm:sup}. In Section~\ref{sec:prelim}, we recall the necessary basics of Sobolev forms and prove estimates for normalized pull-backs. In Section~\ref{sec:Sobev--Poincaré}, we prove Theorem~\ref{thm:limit-extension}. In Section~\ref{sec:proof-of-thm:sup}, we give the proof of Theorem~\ref{thm:sup}. In Section~\ref{sec:submanifold}, we prove Theorem~\ref{thm:submfd}. Lastly, in Section~\ref{sec:example}, we prove Corollaries~\ref{cor:(non-)existence-QRC-type1} and \ref{cor:(non-)existence-QRC-type2}.

\subsection*{Acknowledgements}

The author thanks her thesis advisor Pekka Pankka for all his help improving the manuscript. We also thank Eero Hakavuori and Toni Ikonen for comments and discussions on these topics.

\section{Reduction of Theorem~\ref{thm:main} to Theorem~\ref{thm:sup}} \label{sec:reduction}

In this section, we reduce Theorem~\ref{thm:main} to Theorem~\ref{thm:sup} with the following proposition; see also \cite[Proposition~2.1]{Heikkila-Pankka} and \cite[Section~4]{PR}.

\begin{proposition} \label{prop:reduction}
Let $N$ be a closed, connected, and oriented Riemannian manifold. Let $\omega \in \Omega^n(N)$ be a closed form with $2\le n\le \dim N$ and $0\ne [\omega]_\dR \in \K$. Let $K\ge 1$ and let $F\colon \R^n \to N$ be a $K$-quasiregular $\omega$-curve satisfying $\int_{\R^n} F^* \omega = \infty$ and $\inf_{\R^n} (\norm{\omega} \circ F)>0$. Then there exist a constant $D=D(n,\omega,F)>1$ and a sequence of pairs $(a_j,r_j) \in \R^n \times (0,\infty)$ for which
\[
F\circ T_j \in \F \quad \text{ and } \quad \lim_{j\to \infty} A_\omega(F\circ T_j)=\infty,
\]
where $T_j\colon \R^n \to \R^n$ is the affine map $x\mapsto r_jx+a_j$. In particular,
\[
\sup_{F' \in \F} A_\omega(F') = \infty.
\]
\end{proposition}

The proof of Proposition~\ref{prop:reduction} is based on Rickman's Hunting Lemma \cite[Lemma~5.1]{RI3}. The following formulation is due to Bonk and Poggi-Corradini \cite[Lemma~2.1]{BP}.

\begin{lemma}[Rickman's Hunting Lemma] \label{lem:hunting}
Let $\mu$ be an atomless Borel measure on $\R^n$ satisfying $\mu(\R^n)=\infty$ and $\mu(B)<\infty$ for every ball $B\subset \R^n$. Then there exists a constant $D=D(n)>1$ with the property that, for every $j\in \N$, there exists a ball $B\subset \R^n$ for which
\[
j\le \mu(2B)\le D\mu(B).
\]
\end{lemma}

Proposition~\ref{prop:reduction} is an immediate consequence.

\begin{proof}[Proof of Proposition~\ref{prop:reduction}]
Since $\int_{\R^n} F^* \omega = \infty$, the Borel measure $\mu$ on $\R^n$ defined by $\mu(E)=\int_E F^* \omega$ satisfies the assumptions of Lemma~\ref{lem:hunting}. Hence, there exists a constant $D=D(n)>1$ and a sequence of balls $(B^n(a_j,r_j))$ for which
\[
j\le \int_{B^n(a_j,2r_j)} F^* \omega \le D \int_{B^n(a_j,r_j)} F^* \omega.
\]
Let $D'=D'(n,\omega,F)\ge D$ be a constant for which $\inf_{\R^n} (\norm{\omega} \circ F)\ge (D')^{-1}$. We obtain that the maps $F_j \colon B_2^n \to N$, $x\mapsto F(r_jx+a_j)$, form a sequence in $\cF_{K,D'}(N,\omega)$ satisfying $A_\omega(F_j)\to \infty$. The claim follows.
\end{proof}

The following theorem gives a sufficient condition for a quasiregular curve to have infinite energy. The theorem is a special case of the Bonk--Heinonen type growth result for signed quasiregular curves; see \cite[Theorem 1.1]{Heikkila}.

\begin{theorem}{ \emph{(\cite[Theorem 1.1]{Heikkila})} } \label{thm:signed}
Let $N$ be a connected and oriented Riemannian manifold. Let $\omega \in \Omega^n(N)$ be a closed non-vanishing form with $2\le n\le \dim N$ and
\[
\omega \in \bSpan_\R \left( \bigcup_{k=1}^{n-1} \mathcal{Z}^k(N) \wedge \mathcal{Z}^{n-k}(N) \right),
\]
where $\mathcal{Z}^k(N)$ denotes the space of all smooth, bounded, and closed $k$-forms on $N$. Then every non-constant $K$-quasiregular $\omega$-curve $F\colon \R^n \to N$ satisfies
\[
\liminf_{r\to \infty} \frac{\int_{B^n(0,r)} F^* \omega}{r^\varepsilon} >0,
\] 
where $\varepsilon=\varepsilon(n,K,\omega)>0$.
\end{theorem}

Theorem~\ref{thm:signed} implies trivially that non-constant quasiregular $\omega_\sym^{\wedge m}$-curves $\R^{2m} \to N$, where $\omega_\sym^{\wedge m}$ is as in Example~\ref{ex:symplectic-kunneth}, have infinite energy. It also follows straightforwardly that non-constant quasiregular $\left(\sum_{i=1}^m \pi_i^* \vol_{N_i}\right)$-curves $\R^n \to N_1 \times \cdots \times N_m$, where $\sum_{i=1}^m \pi_i^* \vol_{N_i}$ is as in Example~\ref{ex:product-kunneth}, have infinite energy.

\begin{corollary} \label{cor:signed-product-volume}
Let $n\ge 2$ and let $N_1,\ldots,N_m$ be closed, connected, and oriented Riemannian $n$-manifolds with $H_\dR^*(N_i)\ne H_\dR^*(\bS^n)$. Let $N=N_1 \times \cdots \times N_m$ and $\omega = \sum_{i=1}^m \pi_i^* \vol_{N_i} \in \Omega^n(N)$. Then every non-constant quasiregular $\omega$-curve $F\colon \R^n \to N$ satisfies $\int_{\R^n} F^* \omega = \infty$.
\end{corollary}

\begin{proof}
By \cite[Proposition~2.13]{HIMO}, it holds that
\[
\vol_{N_i} \in \bSpan_\R \left( \bigcup_{k=1}^{n-1} \mathcal{Z}^k(N_i) \wedge \mathcal{Z}^{n-k}(N_i) \right)
\]
for each $i$. Thus,
\[
\omega \in \bSpan_\R \left( \bigcup_{k=1}^{n-1} \mathcal{Z}^k(N) \wedge \mathcal{Z}^{n-k}(N) \right)
\]
and the claim follows by Theorem~\ref{thm:signed}.
\end{proof}

\section{Preliminaries} \label{sec:prelim}

\subsection*{Sobolev forms} \label{sec:prelim-sobolev}

We begin by reviewing the basics of Sobolev spaces of differential forms. A $(k+1)$-form $d\tau \in L_\loc^1(B_2^n; \bigwedge^{k+1} \R^n)$ is the \emph{weak exterior derivative} of a $k$-form $\tau \in L_\loc^1(B_2^n; \bigwedge^k \R^n)$ if
\[
\int_{B_2^n} \tau \wedge d\varphi = (-1)^{k+1} \int_{B_2^n} d\tau \wedge \varphi
\]
for every $\varphi \in \Omega_0^{n-k-1}(B_2^n)$. For $k=0,\ldots,n-1$, $W^{d,\frac{n}{k}}(B_2^n; \bigwedge^k \R^n)$ denotes the Sobolev space of forms $\tau \in L^\frac{n}{k}(B_2^n; \bigwedge^k \R^n)$ with a weak exterior derivative $d\tau \in L^\frac{n}{k}(B_2^n; \bigwedge^{k+1} \R^n)$. For $k=n$, $W^{d,\frac{n}{k}}(B_2^n; \bigwedge^k \R^n)$ denotes $L^1(B_2^n; \bigwedge^n \R^n)$.

Next, we recall the compatibility of the weak exterior derivative and the pull-back of a continuous Sobolev map.

\begin{lemma}{ \emph{(\cite[II.4.7]{RE})} } \label{lem:deriv-of-pullback}
Let $N$ be a smooth manifold and let $F\in W_\loc^{1,n}(B_2^n,N)$ be a continuous map. Then $dF^* \alpha = F^* d\alpha$ for every $\alpha \in \Omega^k(N)$ and $k=0,\ldots,n-1$.
\end{lemma}

We also have an elementary estimate for the the pull-back of a continuous Sobolev map. Namely, if $F\in W_\loc^{1,n}(B_2^n,N)$ is a continuous map into a smooth manifold $N$, then
\begin{equation} \label{eq:pull-back}
\norm{F^* \alpha} \le (\norm{\alpha} \circ F) \norm{DF}^k \text{ a.e.~in } B_2^n
\end{equation}
for $\alpha \in \Omega^k(N)$ and $k=0,\ldots,\dim N$.

Lemma~\ref{lem:deriv-of-pullback} and the pointwise estimate \eqref{eq:pull-back} yield as an immediate consequence the following.

\begin{corollary}
Let $N$ be a closed smooth manifold and let $F\in W^{1,n}(B_2^n,N)$ be a continuous map. Then $F^*$ defines a graded linear map $F^* \colon \Omega^*(N) \to \oplus_{k=0}^n L^\frac{n}{k}(B_2^n; {\bigwedge}^k \R^n)$. In particular, $F^*|_{\mathcal{Z}^k(N)} \colon \mathcal{Z}^k(N) \to W^{d,\frac{n}{k}}(B_2^n; \bigwedge^k \R^n)$ is well-defined for $k=0,\ldots,n$.
\end{corollary}

\subsection*{Estimates for normalized pull-backs}

Next, we define the normalized pull-back of a differential form under a quasiregular curve.

Let $N$ be a closed, connected, and oriented Riemannian manifold and let $\omega \in \Omega^n(N)$ be a closed non-exact form with $2\le n\le \dim N$. Let $K,D\ge 1$. For each $F\in \F$, let $F^! \colon \Omega^*(N) \to \bigoplus_{k=0}^n L^\frac{n}{k}(B_2^n; {\bigwedge}^k \R^n)$ be the linear map
\[
F^! \alpha = A_\omega(F)^{-\frac{k}{n}} F^* \alpha \text{ for } \alpha \in \Omega^k(N) \text{ and } k=0,\ldots,n;
\]
recall that
\[
A_\omega(F) = \int_{B^n} F^* \omega.
\]
We note that $F^!$ commutes with the exterior product and that $F^! \circ h = F^\#$.

The following elementary result gives a bound for the operator norm of $F^!$. The proof is similar to the proof of \cite[Lemma~3.1]{Heikkila-Pankka}, so we omit the details.

\begin{lemma} \label{lem:normalized-bounded}
Let $N$ be a closed, connected, and oriented Riemannian manifold and let $\omega \in \Omega^n(N)$ be a closed non-exact form with $2\le n\le \dim N$. Let $K,D\ge 1$. Let $F\in \F$. Then
\[
\norm{F^! \alpha}_{\frac{n}{k},B_2^n} \le KD^2\norm{\alpha}_\infty
\]
for every $\alpha \in \Omega^k(N)$ and $k=0,\ldots,n$.
\end{lemma}


Lemma~\ref{lem:normalized-bounded} implies that exact forms vanish in a weak sense under limits of operators $F_j^!$ for sequences $(F_j)$ in $\F$ satisfying $A_\omega(F_j)\to \infty$. Since the proof is similar to the proof of \cite[Lemma~3.2]{Heikkila-Pankka}, we omit the details.

\begin{lemma} \label{lem:limit-of-exact}
Let $N$ be a closed, connected, and oriented Riemannian manifold and let $\omega \in \Omega^n(N)$ be a closed non-exact form with $2\le n\le \dim N$. Let $K,D\ge 1$. Let $(F_j)$ be a sequence in $\F$ satisfying $A_\omega(F_j)\to \infty$. Then
\[
\lim_{j\to \infty} \int_{B_2^n} \varphi \wedge F_j^! d\alpha = 0
\]
for $\alpha \in \Omega^k(N)$, $\varphi \in \Omega_0^{n-k-1}(B_2^n)$, and $k=0,\ldots,n-1$.
\end{lemma}


\section{Sobolev--Poincaré limits and their properties} \label{sec:Sobev--Poincaré}

In this section, our end goal is to prove Theorem~\ref{thm:limit-extension}. The proof is based on Sobolev--Poincaré limits.

\subsection*{Motivation and definition}

We begin by recalling the Poincaré operator of Iwaniec and Lutoborski \cite{IL}: \emph{There exists a graded compact linear operator
\[
T\colon \bigoplus_{k=1}^{n-1} L^\frac{n}{k}(B_2^n; {\bigwedge}^k \R^n) \to \bigoplus_{k=1}^{n-1} L^\frac{n}{k}(B_2^n; {\bigwedge}^{k-1} \R^n)
\]
satisfying $\id=dT+Td$. In particular, $dT\omega=\omega$ for every weakly closed form $\omega \in \oplus_{k=1}^{n-1} L^\frac{n}{k}(B_2^n; \bigwedge^k \R^n)$.
}

Next, we show that, for a sequence $(F_j)$ in $\F$, a weak limit of $(F_j^\#)$ is the weak exterior derivative of a limit of $(TF_{j_i}^\#)$ for some subsequence. Before stating a precise formulation, we recall for the reader's convenience that, for $F\in \F$,
\[
F^\# \colon H_\dR^*(N) \to \cW^*(B_2^n), \; H_\dR^k(N) \ni c \mapsto A_\omega(F)^{-\frac{k}{n}} F^*(h(c)) = F^!(h(c)),
\]
where $h\colon H_\dR^*(N) \to \Omega^*(N)$ is the choice of harmonic representative. The proof is similar to the proof of \cite[Lemma~4.1]{Heikkila-Pankka}, so we omit the details.

\begin{lemma} \label{lem:Sob-Poinc}
Let $N$ be a closed, connected, and oriented Riemannian manifold and let $\omega \in \Omega^n(N)$ be a closed non-exact form with $2\le n\le \dim N$. Let $K,D\ge 1$. Let $(F_j)$ be a sequence in $\F$ for which $F_j^\# \weakto L$, where
\[
F_j^\#,L \colon \bigoplus_{k=1}^{n-1} H_\dR^k(N) \to \bigoplus_{k=1}^{n-1} L^\frac{n}{k}(B_2^n; {\bigwedge}^k \R^n).
\]
Then there exists a subsequence $(F_{j_i})$ and a graded linear operator
\[
\widehat{L} \colon \bigoplus_{k=1}^{n-1} H_\dR^k(N) \to \bigoplus_{k=1}^{n-1} W^{d,\frac{n}{k}}(B_2^n; {\bigwedge}^{k-1} \R^n) 
\]
satisfying $TF_{j_i}^\# \to \widehat{L}$. In particular, $d\widehat{L}=L$.
\end{lemma}


We are now ready to give the definition of a Sobolev--Poincaré limit. Let $N$ be a closed, connected, and oriented Riemannian manifold and let $\omega \in \Omega^n(N)$ be a closed non-exact form with $2\le n\le \dim N$. Let $K,D\ge 1$. Let $(F_j)$ be a sequence in $\F$. We say that a graded linear operator
\[
\widehat{L} \colon \bigoplus_{k=1}^{n-1} H_\dR^k(N) \to \bigoplus_{k=1}^{n-1} W^{d,\frac{n}{k}}(B_2^n; {\bigwedge}^{k-1} \R^n)
\]
is a \emph{Sobolev--Poincaré limit of $(F_j^\#)$} if $d\widehat{L}$ is the weak limit of the operators $F_j^\#$ and $\widehat{L}$ is the limit of the operators $TF_j^\#$.

\subsection*{Vague limits of products}

The weak exterior derivative of a Sobolev--Poincaré limit for a sequence whose normalizations grow without bound commutes with the exterior product. We begin by showing the commutativity in a weak sense without assumptions on the growth of the normalizations; see also \cite[Proposition~4.2]{Heikkila-Pankka} and \cite[Lemma~4.2]{PR}. For a precise statement, we recall that a sequence $(\mu_j)$ of Radon measures on $B_2^n$ \emph{converges vaguely to a Radon measure $\mu$ on $B_2^n$} if
\[
\lim_{j\to \infty} \int_{B_2^n} \varphi \, \mathrm{d}\mu_j = \int_{B_2^n} \varphi \, \mathrm{d}\mu
\]
for every $\varphi \in C_0(B_2^n)$.

\begin{proposition} \label{prop:weak-commutativity}
Let $N$ be a closed, connected, and oriented Riemannian manifold and let $\omega \in \Omega^n(N)$ be a closed non-exact form with $2\le n\le \dim N$. Let $K,D\ge 1$. Let $(F_j)$ be a sequence in $\F$ and suppose that the sequence $(F_j^\#)$ has a Sobolev--Poincaré limit $\widehat{L} \colon \oplus_{k=1}^{n-1} H_\dR^k(N) \to \oplus_{k=1}^{n-1} W^{d,\frac{n}{k}}(B_2^n; {\bigwedge}^{k-1} \R^n)$. Let $k_1,k_2\in \{1,\ldots,n-1\}$ for which $k:=k_1+k_2\le n$. Then, for every $c_i \in H_\dR^{k_i}(N)$, we have
\[
F_j^\# c_1 \wedge F_j^\# c_2 \weakto d\widehat{L}(c_1) \wedge d\widehat{L}(c_2).
\]
Here, the weak convergence is the usual weak convergence in $L^\frac{n}{k}(B_2^n; \bigwedge^{k} \R^n)$ if $k<n$ and the vague convergence of measures if $k=n$.
\end{proposition}

\begin{proof}
Let $c_i \in H_\dR^{k_i}(N)$. By Lemma~\ref{lem:normalized-bounded}, the sequence $(F_j^\# c_1 \wedge F_j^\# c_2) = (F_j^! (h(c_1) \wedge h(c_2)))$ is bounded in $L^\frac{n}{k}(B_2^n; \bigwedge^k \R^n)$. Moreover, $\Omega_0^{n-k}(B_2^n)$ is dense in $L^\frac{n}{k}(B_2^n; \bigwedge^k \R^n)$ with respect to the $L^\frac{n}{k}$-norm if $k<n$ and $C_0^\infty(B_2^n)$ is dense in $C_0(B_2^n)$ with respect to the $L^\infty$-norm. Thus, it suffices to prove that
\[
\lim_{j\to \infty} \int_{B_2^n} \varphi \wedge F_j^\# c_1 \wedge F_j^\# c_2 = \int_{B_2^n} \varphi \wedge d\widehat{L}c_1 \wedge d\widehat{L}c_2
\]
for every $\varphi \in \Omega_0^{n-k}(B_2^n)$.

Let $\varphi \in \Omega_0^{n-k}(B_2^n)$. By the triangle inequality, it suffices to estimate
\[
\abs{\int_{B_2^n} \varphi \wedge (d\widehat{L}c_1-F_j^\# c_1) \wedge d\widehat{L}c_2} \text{ and } \abs{\int_{B_2^n} \varphi \wedge F_j^\# c_1 \wedge (d\widehat{L}c_2-F_j^\# c_2)}.
\]
Since
\[
\varphi \wedge d\widehat{L}c_2 \in L^\frac{n}{k_2}(B_2^n; {\bigwedge}^{n-k_1} \R^n) \subset L^\frac{n}{n-k_1}(B_2^n; {\bigwedge}^{n-k_1} \R^n)
\]
and $F_j^\# c_1 \weakto d\widehat{L}c_1$ in $L^\frac{n}{k_1}(B_2^n; \bigwedge^{k_1} \R^n)$, the first term converges to zero.

For the second term, observe that $dTF_j^\# c_2 = F_j^\# c_2$. Thus, we have the estimate
\begin{align*}
\abs{\int_{B_2^n} \varphi \wedge F_j^\# c_1 \wedge (d\widehat{L}c_2-F_j^\# c_2)} &= \abs{\int_{B_2^n} \varphi \wedge F_j^\# c_1 \wedge d(\widehat{L}c_2-TF_j^\# c_2)} \\
&= \abs{\int_{B_2^n} d\varphi \wedge F_j^\# c_1 \wedge (\widehat{L}c_2-TF_j^\# c_2)} \\
&\le C(n) \norm{d\varphi}_\infty \int_{B_2^n} \norm{F_j^\# c_1} \norm{\widehat{L}c_2-TF_j^\# c_2}.
\end{align*}
By Hölder's inequality and Lemma~\ref{lem:normalized-bounded}, we may estimate further by
\begin{align*}
&\int_{B_2^n} \norm{F_j^\# c_1} \norm{\widehat{L}c_2-TF_j^\# c_2} \\
	&\quad \le \norm{F_j^\# c_1}_{\frac{n}{n-k_2},B_2^n} \norm{\widehat{L}c_2 - TF_j^\# c_2}_{\frac{n}{k_2},B_2^n} \\
	&\quad \le C(n,k_1,k_2) \norm{F_j^\# c_1}_{\frac{n}{k_1},B_2^n} \norm{\widehat{L}c_2 - TF_j^\# c_2}_{\frac{n}{k_2},B_2^n} \\
	&\quad \le C(n,k_1,k_2) KD^2 \norm{h(c_1)}_\infty \norm{\widehat{L}c_2 - TF_j^\# c_2}_{\frac{n}{k_2},B_2^n}.
\end{align*}
Since $TF_j^\# c_2 \to \widehat{L}c_2$ in $L^\frac{n}{k_2}(B_2^n; \bigwedge^{k_2} \R^n)$, the claim follows.
\end{proof}

In the next result, we show that the weak commutativity of the weak exterior derivative of a Sobolev--Poincaré limit with the exterior product improves into commutativity, provided that the normalizations grow without bound. The argument in the proof is analogous to the one in the proof of \cite[Lemma~4.3]{Heikkila-Pankka} and we omit the details.

\begin{lemma} \label{lem:commutativity}
Let $N$ be a closed, connected, and oriented Riemannian manifold and let $\omega \in \Omega^n(N)$ be a closed non-exact form with $2\le n\le \dim N$. Let $K,D\ge 1$. Let $(F_j)$ be a sequence in $\F$ for which $A_\omega(F_j)\to \infty$ and suppose that the sequence $(F_j^\#)$ has a Sobolev--Poincaré limit $\widehat{L}$. Then
\[
d\widehat{L}(c_1 \wedge c_2) = d\widehat{L}c_1 \wedge d\widehat{L}c_2
\]
for every $c_1 \in H_\dR^{k_1}(N)$ and $c_2 \in H_\dR^{k_2}(N)$ with $k_1,k_2 \in \{1,\ldots,n-1\}$ satisfying $k_1+k_2<n$.
\end{lemma}


\subsection*{Extension to $\K$}

In the setting of quasiregular mappings in \cite{Heikkila-Pankka}, weak exterior derivatives of Sobolev--Poincare limits are extended to $H_\dR^{\dim N}(N) = \bSpan_\R ([\vol_N]_\dR)$. Next, we show that, in the setting of this paper, the weak exterior derivative of a Sobolev--Poincare limit has a well-defined extension to $\K$ arising from the algebraic structure of $\K$.

\begin{lemma} \label{lem:extension-to-K}
Let $N$ be a closed, connected, and oriented Riemannian manifold and let $\omega \in \Omega^n(N)$ be a closed non-exact form with $2\le n\le \dim N$. Let $K,D\ge 1$. Let $(F_j)$ be a sequence in $\F$ for which $A_\omega(F_j)\to \infty$ and suppose that the sequence $(F_j^\#)$ has a Sobolev--Poincaré limit $\widehat{L}$. Then the mapping
\[
d\widehat{L}_n \colon \K \to L^1(B_2^n; {\bigwedge}^n \R^n), \quad \sum_{i=1}^m c_i \wedge c_i' \mapsto \sum_{i=1}^m d\widehat{L}c_i \wedge d\widehat{L}c_i',
\]
is a well-defined linear map.
\end{lemma}

\begin{proof}
Let $k_1,\ldots,k_m,\ell_1,\ldots,\ell_\nu \in \{1,\ldots,n-1\}$. Let $c_i \in H_\dR^{k_i}(N)$, $c_i' \in H_\dR^{n-k_i}(N)$, $e_i \in H_\dR^{\ell_i}(N)$, and $e_i' \in H_\dR^{n-\ell_i}(N)$ be de Rham classes satisfying $\sum_{i=1}^m c_i \wedge c_i' = \sum_{i=1}^\nu e_i \wedge e_i'$.

Let $\alpha \in \Omega^{n-1}(N)$ be a form satisfying
\[
\sum_{i=1}^m h(c_i) \wedge h(c_i') = \sum_{i=1}^\nu h(e_i) \wedge h(e_i') + d\alpha
\]
and let $\varphi \in C_0^\infty(B_2^n)$. By Proposition~\ref{prop:weak-commutativity}, we obtain
\begin{align*}
\int_{B_2^n} \varphi \left( \sum_{i=1}^m d\widehat{L}c_i \wedge d\widehat{L}c_i' \right) &= \sum_{i=1}^m \lim_{j\to \infty} \int_{B_2^n} \varphi F_j^\# c_i \wedge F_j^\# c_i' \\
&= \lim_{j\to \infty} \sum_{i=1}^m \int_{B_2^n} \varphi F_j^\# c_i \wedge F_j^\# c_i' \\
&= \lim_{j\to \infty} \int_{B_2^n} \varphi \sum_{i=1}^m F_j^!( h(c_i) \wedge h(c_i')) \\
&= \lim_{j\to \infty} \int_{B_2^n} \varphi F_j^! \left( \sum_{i=1}^m h(c_i) \wedge h(c_i') \right).
\end{align*}
A similar computation yields that
\[
\int_{B_2^n} \varphi \left( \sum_{i=1}^\nu d\widehat{L}e_i \wedge d\widehat{L}e_i' \right) = \lim_{j\to \infty} \int_{B_2^n} \varphi F_j^! \left( \sum_{i=1}^\nu h(e_i) \wedge h(e_i') \right).
\]
Since, by Lemma~\ref{lem:limit-of-exact}, we have
\begin{align*}
&\abs{ \int_{B_2^n} \varphi F_j^! \left( \sum_{i=1}^m h(c_i) \wedge h(c_i') \right) - \int_{B_2^n} \varphi F_j^! \left( \sum_{i=1}^\nu h(e_i) \wedge h(e_i') \right) } \\
	&\quad = \abs{ \int_{B_2^n} \varphi F_j^! d\alpha } \to 0,
\end{align*}
it follows that
\[
\sum_{i=1}^m d\widehat{L}c_i \wedge d\widehat{L}c_i' = \sum_{i=1}^\nu d\widehat{L}e_i \wedge d\widehat{L}e_i'.
\]
Hence, $d\widehat{L}_n$ is well-defined.
\end{proof}

The following proposition states that extending $d\widehat{L}$ with $d\widehat{L}_n \colon \K \to L^1(B_2^n; \bigwedge^n \R^n)$ is compatible with vague convergence of measures on $B_2^n$; see also \cite[Theorem~4.4]{Heikkila-Pankka}.

\begin{proposition} \label{prop:limit-measure}
Let $N$ be a closed, connected, and oriented Riemannian manifold and let $\omega \in \Omega^n(N)$ be a closed non-exact form with $2\le n\le \dim N$. Let $K,D\ge 1$. Let $(F_j)$ be a sequence in $\F$ for which $A_\omega(F_j)\to \infty$ and suppose that the sequence $(F_j^\#)$ has a Sobolev--Poincaré limit $\widehat{L}$. Let $\alpha \in \Omega^n(N)$ be a closed form satisfying $[\alpha]_\dR \in \K$. Then
\[
\int_{B_2^n} \varphi d\widehat{L}_n [\alpha]_\dR = \lim_{j\to \infty} \int_{B_2^n} \varphi F_j^! \alpha
\]
for every $\varphi \in C_0(B_2^n)$. In particular, $F_j^! \alpha \weakto d\widehat{L}_n [\alpha]_\dR$ vaguely as measures and, if each $F_j^! \alpha$ is non-negative, then $d\widehat{L}_n [\alpha]_\dR$ is non-negative.
\end{proposition}

\begin{proof}
Let $[\alpha]_\dR = \sum_{i=1}^m c_i \wedge c_i'$ be a representation of $[\alpha]_\dR$ in $\K$ and let $\beta \in \Omega^{n-1}(N)$ be a form satisfying
\[
\sum_{i=1}^m h(c_i) \wedge h(c_i') = \alpha + d\beta.
\]
Since $\norm{F_j^! \alpha}_{1,B_2^n} \le KD^2 \norm{\alpha}_\infty$ by Lemma~\ref{lem:normalized-bounded} and $C_0^\infty(B_2^n)$ is dense in $C_0(B_2^n)$ with respect to the $L^\infty$-norm, it suffices to prove that
\[
\int_{B_2^n} \varphi d\widehat{L}_n [\alpha]_\dR = \lim_{j\to \infty} \int_{B_2^n} \varphi F_j^! \alpha
\]
for every $\varphi \in C_0^\infty(B_2^n)$.

By Proposition~\ref{prop:weak-commutativity} and Lemma~\ref{lem:limit-of-exact}, we have
\begin{align*}
&\int_{B_2^n} \varphi d\widehat{L}_n [\alpha]_\dR \\
	&\quad = \int_{B_2^n} \varphi \left( \sum_{i=1}^m d\widehat{L}c_i \wedge d\widehat{L}c_i' \right) = \sum_{i=1}^m \lim_{j\to \infty} \int_{B_2^n} \varphi F_j^\# c_i \wedge F_j^\# c_i' \\
	&\quad = \lim_{j\to \infty} \int_{B_2^n} \varphi \sum_{i=1}^m F_j^! (h(c_i) \wedge h(c_i')) = \lim_{j\to \infty} \int_{B_2^n} \varphi F_j^! (\alpha + d\beta) \\
	&\quad = \lim_{j\to \infty} \int_{B_2^n} \varphi F_j^! \alpha
\end{align*}
for every $\varphi \in C_0^\infty(B_2^n)$. This concludes the proof.
\end{proof}

For $\alpha=\omega$, we get the following improved version of Proposition~\ref{prop:limit-measure}; see also \cite[Corollary~4.5]{Heikkila-Pankka}.

\begin{corollary} \label{cor:prob-measure}
Let $N$ be a closed, connected, and oriented Riemannian manifold. Let $2\le n\le \dim N$ and let $\omega \in \Omega^n(N)$ be a closed form with $0\ne [\omega]_\dR \in \K$. Let $K,D\ge 1$. Let $(F_j)$ be a sequence in $\F$ for which $A_\omega(F_j)\to \infty$ and suppose that the sequence $(F_j^\#)$ has a Sobolev--Poincaré limit $\widehat{L}$. Then there exists a non-negative measure $\mu \in L^1(B_2^n; \bigwedge^n \R^n)$ having following properties:
\begin{enumerate}
\item $\mu = d\widehat{L}_n [\omega]_\dR$, where $d\widehat{L}_n$ is the extension of $d\widehat{L}$ to $\K$ as defined in Lemma~\ref{lem:extension-to-K},
\item $F_j^! \omega \weakto \mu$ vaguely as measures on $B_2^n$,
\item $F_j^! \omega \llcorner \bar B^n \weakto \mu \llcorner \bar B^n$ weakly as measures on $\bar B^n$,
and
\item $\mu(B^n)=1$.
\end{enumerate}
\end{corollary}

\begin{remark}
The assumption $[\omega]_\dR \in \K$ is necessary; see e.g.~the example in \cite[Section~8]{Heikkila-Pankka}.
\end{remark}

\begin{proof}[Proof of Corollary~\ref{cor:prob-measure}]
By Proposition~\ref{prop:limit-measure}, there exists a non-negative measure $\mu \in L^1(B_2^n; \bigwedge^n \R^n)$ satisfying (1) and (2). Thus, it suffices to show that $\mu$ satisfies also (3) and (4).

Let $\varphi \in C(\bar B^n)$ be a non-negative function. Let $\tilde{\varphi} \in C_b(B_2^n)$ be a non-negative bounded function satisfying $\tilde{\varphi}|_{\bar B^n}=\varphi$. Then $\tilde{\varphi} F_j^! \omega \weakto \tilde{\varphi} \mu$ vaguely as measures on $B_2^n$.

Let $0<r<1<R<2$. Let $\psi_r \in C_0(B^n)$ and $\psi_R \in C_0(B^n(0,R))$ satisfy $\chi_{B^n(0,r)} \le \psi_r \le 1$ and $\chi_{B^n} \le \psi_R \le 1$. Since $\tilde{\varphi} \mu$ is non-negative, we obtain
\[
\int_{\bar B^n(0,r)} \tilde{\varphi} \mu \le \int_{B_2^n} \psi_r \tilde{\varphi} \mu = \lim_{j\to \infty} \int_{B_2^n} \psi_r \tilde{\varphi} F_j^! \omega \le \liminf_{j\to \infty} \int_{B^n} \tilde{\varphi} F_j^! \omega
\]
and
\[
\int_{B^n(0,R)} \tilde{\varphi} \mu \ge \int_{B_2^n} \psi_R \tilde{\varphi} \mu = \lim_{j\to \infty} \int_{B_2^n} \psi_R \tilde{\varphi} F_j^! \omega \ge \limsup_{j\to \infty} \int_{B^n} \tilde{\varphi} F_j^! \omega.
\]
Since $\mu \in L^1(B_2^n; \bigwedge^n \R^n)$,
\[
\liminf_{j\to \infty} \int_{B^n} \tilde{\varphi} F_j^! \omega \ge \tilde{\varphi} \mu(B^n) = \tilde{\varphi} \mu(\bar B^n) \ge \limsup_{j\to \infty} \int_{B^n} \tilde{\varphi} F_j^! \omega.
\]
Thus
\begin{equation}
\label{eq:weak-convergence-for-phi}
\lim_{j\to \infty} \int_{\bar B^n} \varphi F_j^! \omega = \int_{\bar B^n} \varphi \mu.    
\end{equation}
Equality \eqref{eq:weak-convergence-for-phi} is then obtained for an arbitrary function $\varphi \in C(\bar B^n)$ by writing $\varphi = \max(\varphi,0) - \max(-\varphi,0)$.

Especially, since $\mu \in L^1(B_2^n; \bigwedge^n \R^n)$, we have, for $\varphi = 1$, that
\begin{align*}
\mu(B^n) &= \int_{\bar B^n} \mu = \lim_{j\to \infty} \int_{\bar B^n} F_j^! \omega = \lim_{j\to \infty} \int_{B^n} F_j^! \omega \\
&= \lim_{j\to \infty} A_\omega(F_j)^{-1} \int_{B^n} F_j^* \omega = 1.
\end{align*}
This concludes the proof.
\end{proof}

\subsection*{Homomorphism $H_\dR^*(N) \to \cW^*(B_2^n)$}

Next, we show that, if the normalizations grow without bound, then the weak exterior derivative of a Sobolev--Poincare limit extends to a homomorphism $H_\dR^*(N) \to \cW^*(B_2^n)$; see also \cite[Proposition~5.1]{Heikkila-Pankka}.

For a precise statement, we introduce the following notation. Let $N$ be a closed, connected, and oriented Riemannian manifold and let $\omega \in \Omega^n(N)$ be a closed non-exact form with $2\le n\le \dim N$. Let $\pi \colon H_\dR^n(N) \to H_\dR^n(N)$ be a projection onto $\K$, i.e., a linear map satisfying $\pi \circ \pi = \pi$ and $\im \pi = \K$.

Let $K,D\ge 1$. Let $(F_j)$ be a sequence in $\F$ for which $A_\omega(F_j)\to \infty$ and let $\widehat{L}$ be the Sobolev--Poincaré limit of $(F_j^\#)$. Let
\[
\Psi_{d\widehat{L}}^\pi \colon H_\dR^*(N) \to \cW^*(B_2^n)
\]
be the graded linear map defined by
\[
H_\dR^k(N) \ni c\mapsto \begin{cases}
\text{constant function } h(c), &\, \text{if } k=0, \\
d\widehat{L}c, &\, \text{if } 1\le k\le n-1, \\
d\widehat{L}_n\pi c, &\, \text{if } k=n, \\
0, &\, \text{otherwise.}
\end{cases}
\]
For the reader's convenience, we recall that $d\widehat{L}_n \colon \K \to L^1(B_2^n; \bigwedge^n \R^n)$ is the linear map
\[
\sum_{i=1}^m c_i \wedge c_i' \mapsto \sum_{i=1}^m d\widehat{L}c_i \wedge d\widehat{L}c_i'.
\]

\begin{proposition} \label{prop:extension-homom}
Let $N$ be a closed, connected, and oriented Riemannian manifold and let $\omega \in \Omega^n(N)$ be a closed non-exact form with $2\le n\le \dim N$. Let $\pi \colon H_\dR^n(N) \to H_\dR^n(N)$ be a projection onto $\K$. Let $K,D\ge 1$. Let $(F_j)$ be a sequence in $\F$ for which $A_\omega(F_j)\to \infty$ and suppose that the sequence $(F_j^\#)$ has a Sobolev--Poincaré limit $\widehat{L}$. Then $\Psi_{d\widehat{L}}^\pi \colon H_\dR^*(N) \to \cW^*(B_2^n)$ is a homomorphism.
\end{proposition}

\begin{proof}
Let $k_1,k_2 \in \{0,\ldots,\dim N\}$ and let $c_i \in H_\dR^{k_i}(N)$. Without loss of generality, we may assume that $k_1,k_2>0$ and that $k_1+k_2 \le n$.

If $k_1+k_2<n$, then, by Lemma~\ref{lem:commutativity}, we have
\[
\Psi_{d\widehat{L}}^\pi(c_1 \wedge c_2) = d\widehat{L}(c_1 \wedge c_2) = d\widehat{L}c_1 \wedge d\widehat{L}c_2 = \Psi_{d\widehat{L}}^\pi c_1 \wedge \Psi_{d\widehat{L}}^\pi c_2.
\]
Suppose $k_1+k_2=n$. Since $\pi|_{\K}=\id$, we obtain
\[
\Psi_{d\widehat{L}}^\pi (c_1 \wedge c_2) = d\widehat{L}_n\pi (c_1 \wedge c_2) = d\widehat{L}_n(c_1 \wedge c_2) = d\widehat{L}c_1 \wedge d\widehat{L}c_2 = \Psi_{d\widehat{L}}^\pi c_1 \wedge \Psi_{d\widehat{L}}^\pi c_2.
\]
The claim follows by linearity.
\end{proof}

We are now ready to prove the following generalization of Theorem~\ref{thm:limit-extension}.

\begin{theorem} \label{thm:full-limit-extension}
Let $N$ be a closed, connected, and oriented Riemannian manifold and let $\omega \in \Omega^n(N)$ be a closed non-exact form with $2\le n\le \dim N$. Let $K,D\ge 1$. Let $(F_j)$ be a sequence in $\F$ for which $A_\omega(F_j)\to \infty$ and $F_j^\# \weakto L$, where
\[
F_j^\#,L \colon \bigoplus_{k=1}^{n-1} H_\dR^k(N) \to \bigoplus_{k=1}^{n-1} L^\frac{n}{k}(B_2^n; {\bigwedge}^k \R^n).
\]
Then $L$ extends to a graded algebra homomorphism $L\colon H_\dR^*(N)\to \cW^*(B_2^n)$ for which $F_j^\#|_{\K} \weakto L|_{\K}$ vaguely as measures. Moreover, if $[\omega]_\dR \in \K$, then
\[
\int_{B^n} L[\omega]_\dR = 1 \; \text{ and } \; Lc\ne 0 \text{ if } c \text{ is a factor of } [\omega]_\dR,
\]
i.e., $c\in H_\dR^*(N)$ and there exists $c'\in H_\dR^{n-*}(N)$ for which $c\wedge c' = [\omega]_\dR$.
\end{theorem}

\begin{proof}
By Lemma~\ref{lem:Sob-Poinc}, there exists a subsequence $(F_{j_i})$ for which the operator $L$ is the weak exterior derivative $d\widehat{L}$ of a Sobolev--Poincaré limit $\widehat{L}$ of $(F_{j_i}^\#)$. Let $\pi \colon H_\dR^n(N) \to H_\dR^n(N)$ be a projection onto $\K$. Then, by Proposition~\ref{prop:extension-homom}, the map $\Psi_{d\widehat{L}}^\pi \colon H_\dR^*(N) \to \cW^*(B_2^n)$ is a graded algebra homomorphism and $\Psi_{d\widehat{L}}^\pi$ extends $L$ by definition. Furthermore, for each $c\in \K$, we have
\[
F_j^\# c = F_j^!(h(c)) \weakto d\widehat{L}_nc = d\widehat{L}_n \pi c = \Psi_{d\widehat{L}}^\pi c
\]
vaguely as measures by Proposition~\ref{prop:limit-measure}.

Suppose that $[\omega]_\dR \in \K$. By Corollary~\ref{cor:prob-measure}, we obtain
\[
\int_{B^n} \Psi_{d\widehat{L}}^\pi [\omega]_\dR = \int_{B^n} d\widehat{L}_n\pi [\omega]_\dR = \int_{B^n} d\widehat{L}_n [\omega]_\dR = 1.
\]
Hence, it remains to prove that $\Psi_{d\widehat{L}}^\pi c\ne 0$ if $c$ is a factor of $[\omega]_\dR$.

Let $c\in H_\dR^k(N)$ and $c'\in H_\dR^{n-k}(N)$ satisfy $c\wedge c'=[\omega]_\dR$, where $k\in \{0,\ldots,n\}$. Then, we have
\[
\Psi_{d\widehat{L}}^\pi c\wedge \Psi_{d\widehat{L}}^\pi c' = \Psi_{d\widehat{L}}^\pi (c\wedge c') = \Psi_{d\widehat{L}}^\pi [\omega]_\dR \ne 0
\]
implying $\Psi_{d\widehat{L}}^\pi c\ne 0$.
\end{proof}

\section{Proof of Theorem~\ref{thm:sup}} \label{sec:proof-of-thm:sup}

In this section, we recall the statement of Theorem~\ref{thm:sup} and give its proof.

\begin{named}{Theorem~\ref{thm:sup}}
Let $N$ be a closed, connected, and oriented Riemannian manifold. Let $2\le n\le \dim N$ and let $\omega \in \Omega^n(N)$ be a closed form with $0\ne [\omega]_{\dR} \in \K$. Assume also that there exists $K,D\ge 1$ for which
\[
\sup_{F\in \F} A_\omega(F) = \infty.
\]
Then there exists a graded algebra homomorphism $\Phi \colon H_{\dR}^*(N) \to \bigwedge^* \R^n$ for which $\Phi [\omega]_{\dR}\ne 0$.
\end{named}

Before giving the proof of Theorem~\ref{thm:sup}, we state the following elementary lemma.

\begin{lemma} \label{lem:existence-of-weak-limit}
Let $N$ be a closed, connected, and oriented Riemannian manifold and let $\omega \in \Omega^n(N)$ be a closed non-exact form with $2\le n\le \dim N$. Let $K,D\ge 1$. Let $(F_j)$ be a sequence in $\F$. Then there exist a subsequence $(F_{j_i})$ and an operator
\[
L\colon \bigoplus_{k=1}^{n-1} H_\dR^k(N) \to \bigoplus_{k=1}^{n-1} L^\frac{n}{k}(B_2^n,{\bigwedge}^k \R^n)
\]
satisfying $F_{j_i}^\#|_{H_\dR^k(N)} \weakto L$ for $k=1,\ldots,n-1$.
\end{lemma}

\begin{proof}
Since each $H_\dR^k(N)$ is finite-dimensional, the claim follows immediately by Lemma~\ref{lem:normalized-bounded} and the Banach--Alaoglu theorem.
\end{proof}

We are now ready to prove Theorem~\ref{thm:sup}.

\begin{proof}[Proof of Theorem~\ref{thm:sup}]
There exist a sequence $(F_j)$ in $\F$ and an operator
\[
L\colon \bigoplus_{k=1}^{n-1} H_\dR^k(N) \to \bigoplus_{k=1}^{n-1} L^\frac{n}{k}(B_2^n,{\bigwedge}^k \R^n)
\]
satisfying $A_\omega(F_j)\to \infty$ and $F_j^\#|_{\oplus_{k=1}^{n-1} H_\dR^k(N)} \weakto L$ by Lemma~\ref{lem:existence-of-weak-limit}. By Theorem~\ref{thm:full-limit-extension}, the operator $L$ extends to a graded algebra homomorphism $L\colon H_\dR^*(N)\to \cW^*(B_2^n)$ satisfying $\int_{B^n} L[\omega]_\dR = 1$.

Let $c_1,\ldots,c_m,c_{m+1},\ldots,c_\nu$ be a basis of $H_\dR^*(N)$, where $c_1,\ldots,c_m$ is a basis of $\K$. Since $[\omega]_\dR \in \K$, there exist $\lambda_1,\ldots,\lambda_m \in \R$ satisfying $[\omega]_\dR = \sum_{i=1}^m \lambda_i c_i$.

We fix point-wise Borel representatives for $Lc_i$ and $L(c_i \wedge c_\ell)$, where $i,\ell=1,\ldots,\nu$. Let
\[
E = \left\{ x\in B^n \colon \left( \sum_{i=1}^m \lambda_i Lc_i \right)(x) \ne 0 \right\}
\]
and
\[
E_{i,\ell} = \{ x\in B^n \colon (Lc_i)(x) \wedge (Lc_\ell)(x) = (Lc_i \wedge c_\ell)(x) \}
\]
for $i,\ell=1,\ldots,\nu$. Now $m_n(E_{i,\ell})=m_n(B^n)$ for each pair $(i,\ell)\in \{1,\ldots,\nu\}^2$. Furthermore, $m_n(E)>0$ since
\[
\int_E \sum_{i=1}^m \lambda_i Lc_i = \int_{B^n} \sum_{i=1}^m \lambda_i Lc_i = \int_{B^n} L[\omega]_\dR = 1.
\]
Thus there exists $x_0 \in \cap_{i,\ell=1}^\nu E\cap E_{i,\ell}$.

Let $\Phi \colon H_\dR^*(N)\to \bigwedge^* \R^n$ be the linear map defined by $c_i \mapsto (Lc_i)(x_0)$ for $i=1,\ldots,\nu$. By definition, the map $\Phi$ is graded. Moreover, we have
\[
\Phi(c_i \wedge c_\ell) = (L(c_i \wedge c_\ell))(x_0) = (Lc_i)(x_0) \wedge (Lc_\ell)(x_0) = \Phi c_i \wedge \Phi c_\ell
\]
for each pair $(i,\ell)\in \{1,\ldots,\nu\}^2$ since $x_0 \in E_{i,\ell}$. Hence, the map $\Phi$ is a homomorphism by linearity. Finally,
\[
\Phi [\omega]_\dR = \Phi \left( \sum_{i=1}^m \lambda_i c_i \right) = \sum_{i=1}^m \lambda_i \Phi c_i = \left( \sum_{i=1}^m \lambda_i Lc_i \right)(x_0) \ne 0
\]
since $x_0 \in E$.
\end{proof}

\section{Submanifolds controlled by quasiregular curves} \label{sec:submanifold}

In this section, we present two results, which together yield a stronger version of Theorem~\ref{thm:submfd}. The first result is, in fact, a reformulation of Theorem~\ref{thm:main}.

\begin{theorem} \label{thm:reformulated-main}
Let $N$ be a closed, connected, and oriented Riemannian manifold. Let $2\le n\le \dim N$ and let $\omega \in \Omega^n(N)$ be a closed form satisfying $0\ne [\omega]_\dR \in \K$. Let $K\ge 1$ and let $F\colon \R^n \to N$ be a $K$-quasiregular $\omega$-curve for which $\int_{\R^n} F^* \omega = \infty$ and $\inf_{\R^n} (\norm{\omega} \circ F)>0$. Then there exists a graded algebra homomorphism $\Psi \colon H_\dR^*(N) \to \cW^*(B_2^n)$ satisfying
\[
\int_{B^n} \Psi [\omega]_\dR = 1 \quad \text{ and } \quad \core_\omega(F) \cap \K \subset \ker \Psi.
\]
\end{theorem}

Here and in what follows, $\core_\omega(F) \subset H_\dR^n(N)$ denotes the subspace of de Rham classes $c$ satisfying
\[
\lim_{j\to \infty} \frac{\int_{B^n(a_j,r_j)} |F^*(h(c))|}{\int_{B^n(a_j,r_j)} F^* \omega} = 0
\]
for any sequence $(B^n(a_j,r_j))$ with $\int_{B^n(a_j,r_j)} F^* \omega \to \infty$.

\begin{proof}
By Proposition~\ref{prop:reduction}, there exist $D\ge 1$ and $(a_j,r_j)\in \R^n \times (0,\infty)$ for which $F_j:=F\circ T_j \in \F$ and $A_\omega(F_j)\to \infty$, where $T_j \colon \R^n \to \R^n$ is the affine map $x\mapsto r_jx + a_j$. By passing to a subsequence, Lemma~\ref{lem:existence-of-weak-limit} and Theorem~\ref{thm:full-limit-extension} together yield that there exists a graded algebra homomorphism $L\colon H_\dR^*(N) \to \cW^*(B_2^n)$ for which  $F_j^\#|_{\K} \weakto L|_{\K}$ vaguely as measures and
\[
\int_{B^n} L[\omega]_\dR = 1.
\]
It remains to show that $\core_\omega (F) \cap \K \subset \ker L$.

Let $c\in \core_\omega(F) \cap \K$ and let $\varphi \in C_0(B_2^n)$. Then
\[
\int_{B_2^n} \varphi Lc = \lim_{j\to \infty} \int_{B_2^n} \varphi F_j^\# c = \lim_{j\to \infty} \int_{B_2^n} \varphi F_j^!(h(c)).
\]
On the other hand, since $F_j \in \F$ and $c\in \core_\omega(F)$, we have
\begin{align*}
\abs{\int_{B_2^n} \varphi F_j^!(h(c))} &= A_\omega(F_j)^{-1} \abs{\int_{B_2^n} \varphi F_j^*(h(c))} \\
&\le A_\omega(F_j)^{-1} \norm{\varphi}_\infty \int_{B_2^n} \abs{F_j^*(h(c))} \\
&= \norm{\varphi}_\infty \frac{\int_{B^n(a_j,2r_j)} \abs{F^*(h(c))}}{\int_{B^n(a_j,r_j)} F^* \omega} \\
&\le D\norm{\varphi}_\infty \frac{\int_{B^n(a_j,2r_j)} \abs{F^*(h(c))}}{\int_{B^n(a_j,2r_j)} F^* \omega} \to 0.
\end{align*}
Hence
\[
\int_{B_2^n} \varphi Lc = 0
\]
for an arbitrary function $\varphi \in C_0(B_2^n)$, which yields that $Lc = 0$. This concludes the proof.
\end{proof}

The next result is an elementary algebraic result based on the Poincaré duality on a submanifold.

\begin{proposition} \label{prop:algebraic-submanifold}
Let $N$ be a closed, connected, and oriented Riemannian manifold and let $M\subset N$ be a closed, connected, and oriented $n$-submanifold with $n\ge 2$. Let $c\in \K$ and let $\Psi \colon H_\dR^*(N) \to \cW^*(B_2^n)$ be a graded algebra homomorphism for which
\[
\int_{B^n} \Psi c = 1 \quad \text{ and } \quad \ker \iota^* \cap \K \subset \ker \Psi,
\]
where $\iota \colon M\to N$ is the inclusion. Let $1\le k\le n-1$ be an index for which $\iota^* \colon H_\dR^{n-k}(N) \to H_\dR^{n-k}(M)$ is surjective. Then $\dim \iota^* H_\dR^k(N) \le \binom{n}{k}$.
\end{proposition}

As a preparatory step, we recall the following elementary dimension bound for subspaces of $\cW^*(B_2^n)$; see e.g.~\cite[Corollary~4.4]{PR} for a proof.

\begin{lemma} \label{lem:dim-bound-subalgebra}
Let $n\ge 2$ and $1\le k\le n-1$. Let $u_1,\ldots,u_m \in \cW^k(B_2^n)$ and $v_1,\ldots,v_m \in \cW^{n-k}(B_2^n)$ satisfy $u_i \wedge v_\ell = \delta_{i\ell} w$, where $\delta_{i\ell}$ is the Kronecker delta and $0\ne w\in L^1(B_2^n; \bigwedge \R^n)$. Then $m\le \dim \bigwedge^k \R^n$.
\end{lemma}

\begin{proof}[Proof of Proposition~\ref{prop:algebraic-submanifold}]
Let $\iota^* c_1,\ldots,\iota^* c_m$ be a basis of $\iota^* H_\dR^k(N)$. By Poincaré duality, there exist de Rham classes $\iota^* c_1',\ldots,\iota^* c_m' \in H_\dR^{n-k}(M)$ satisfying $\iota^* c_i \wedge \iota^* c_\ell' = \delta_{i\ell} [\vol_M]_\dR$, where $\delta_{i\ell}$ is the Kronecker delta.

Since $\iota^* c\ne 0$, we may fix $\lambda \ne 0$ for which $\lambda \iota^* c = [\vol_M]_\dR$. Then $c_i \wedge c_\ell' - \delta_{i\ell} \lambda c \in \ker \iota^* \cap \, \K$. Thus
\[
\Psi c_i \wedge \Psi c_\ell' = \Psi (c_i \wedge c_\ell') = \Psi (\delta_{i\ell} \lambda c) = \delta_{i\ell} \lambda \Psi c,
\]
which yields $m\le \dim \bigwedge^k \R^n = \binom{n}{k}$ by Lemma~\ref{lem:dim-bound-subalgebra}.
\end{proof}

Combining Theorem~\ref{thm:reformulated-main} and Proposition~\ref{prop:algebraic-submanifold} yields immediately the following variation of Theorem~\ref{thm:submfd}.

\begin{theorem} \label{thm:submfd-with-core}
Let $N$ be a closed, connected, and oriented Riemannian manifold. Let $2\le n\le \dim N$ and let $\omega \in \Omega^n(N)$ be a closed form for which $0\ne [\omega]_\dR \in \K$. Let $F\colon \R^n \to N$ be a quasiregular $\omega$-curve for which $\int_{\R^n} F^* \omega = \infty$ and $\inf_{\R^n} (\norm{\omega} \circ F)>0$. Let $M\subset N$ be a closed, connected, and oriented Riemannian $n$-submanifold satisfying $\iota^*(H_\dR^*(N))=H_\dR^*(M)$ and $\ker \iota^* \cap \, \K \subset \core_\omega(F)$, where $\iota \colon M\to N$ is the inclusion. Then $\dim H_\dR^k(M) \le \binom{n}{k}$ for $k=0,\ldots,n$.
\end{theorem}

The following example shows that the assumption $\ker \iota^* \cap \, \K \subset \core_\omega(F)$ cannot be replaced by a weaker assumption $\iota^* [\omega]_\dR \ne 0$ in Theorem~\ref{thm:submfd-with-core}.

\begin{example} \label{ex:constant-in-one-factor}
Let $n\ge 2$. Let $M$ and $P$ be closed, connected, and oriented Riemannian manifolds with $\dim M=n$ and $\dim P\ge n$. Let $\xi \in \Omega^n(P)$ be a closed non-exact form. Let $f\colon \R^n \to M$ be a constant map and let $g\colon \R^n \to P$ be a quasiregular $\xi$-curve satisfying $\int_{\R^n} g^* \xi = \infty$.

Let $N=M\times P$, $\omega = \pi_M^* \vol_M + \pi_P^* \xi$, and $F=(f,g)\colon \R^n \to N$. Then $F$ is a quasiregular $\omega$-curve with $\int_{\R^n} F^* \omega = \infty$.

Let also $\iota \colon M\to N$ be an inclusion $m\mapsto (m,p)$ for some fixed $p\in P$. Then $\iota^* \colon H_\dR^*(N) \to H_\dR^*(M)$ is surjective and $\iota^* [\omega]_\dR = [\vol_M]_\dR \ne 0$.
\end{example}

Compare now Example~\ref{ex:constant-in-one-factor} with the choices $M=\#^2 (\bS^1 \times \bS^1)$, $P=\bS^1 \times \bS^1$, $\xi = \vol_P$, and $g(x,y)=(e^{2\pi ix},e^{2\pi iy})$ to Theorem~\ref{thm:submfd-with-core}.

Considering further the case, where the ambient manifold is a product manifold and the studied submanifold is one of the factors, we obtain the following result.

\begin{corollary}
Let $n\ge 2$. Let $M$ and $P$ be closed, connected, and oriented Riemannian manifolds with $\dim M=n$ and $\dim P\ge n$. Let $\xi \in \Omega^n(P)$ be a closed form for which $[\xi]_\dR \in K^n(P)$. Let $N=M\times P$ and let $\omega = \pi_M^* \vol_M + \pi_P^* \xi$. Let $F\colon \R^n \to N$ be a quasiregular $\omega$-curve for which $\int_{\R^n} F^* \omega = \infty$ and $[\pi_P^* \xi]_\dR \in \core_\omega(F)$. Then there exists a graded algebra homomorphism $\Psi \colon H_\dR^*(M) \to \cW^*(B_2^n)$ satisfying
\[
\int_{B^n} \Psi [\vol_M]_\dR = 1.
\]
\end{corollary}

\begin{proof}
We may assume that $H_\dR^*(M)\ne H_\dR^*(\bS^n)$. Then, by Theorem \ref{thm:reformulated-main}, there exists a graded algebra homomorphism $\Psi \colon H_\dR^*(N) \to \cW^*(B_2^n)$ satisfying
\[
\int_{B^n} \Psi [\omega]_\dR = 1 \quad \text{ and } \quad \core_\omega(F) \cap \K \subset \ker \Psi.
\]
Since $\Psi[\pi_P^* \xi]_\dR=0$, it follows that $\Psi \circ \pi_M^* \colon H_\dR^*(M) \to \cW^*(B_2^n)$ is a graded algebra homomorphism satisfying
\[
\int_{B^n} (\Psi \circ \pi_M^*)[\vol_M]_\dR = \int_{B^n} \Psi [\omega]_\dR = 1.
\]
\end{proof}

\section{Examples of pairs $(N,\omega)$ which are not quasiregularly elliptic} \label{sec:example}

In this section, we show that Theorem~\ref{thm:main} yields obstructions for non-constant quasiregular curves from the Euclidean space to product manifolds containing $\C P^2$ with respect to product forms containing the standard symplectic form on $\C P^2$. In particular, for pairs $(S_\nu \times \C P^2,\pi_1^* \vol_{S_\nu} \wedge \pi_2^* \omega_\sym)$ we obtain an obstruction for $\nu \ge 8$.

\begin{proposition} \label{prop:prod-obstruction-type2}
Let $\nu \ge 8$ and let $S_\nu = \#^\nu(\bS^2 \times \bS^2)$. Let $\omega_\sym \in \Omega^2(\C P^2)$ be the standard symplectic form. Let $F\colon \R^6 \to S_\nu \times \C P^2$ be a quasiregular curve with respect to $\omega_{S_\nu}=\pi_1^* \vol_{S_\nu} \wedge \pi_2^* \omega_\sym$. Then $F$ is constant.
\end{proposition}

Proposition~\ref{prop:prod-obstruction-type2} yields Corollary~\ref{cor:(non-)existence-QRC-type2} as an immediate consequence.

On the other hand, for pairs $(T_g \times \C P^2,\pi_1^* \vol_{T_g} \wedge \pi_2^* \omega_\sym)$ we obtain an obstruction for $g\ge 2$.

\begin{proposition} \label{prop:prod-obstruction-type1}
Let $g\ge 2$ and let $T_g = \#^g(\bS^1 \times \bS^1)$. Let $\omega_\sym \in \Omega^2(\C P^2)$ be the standard symplectic form. Let $F\colon \R^4 \to T_g \times \C P^2$ be a quasiregular curve with respect to $\omega_{T_g}=\pi_1^* \vol_{T_g} \wedge \pi_2^* \omega_\sym$. Then $F$ is constant.
\end{proposition}

The obstruction in Proposition~\ref{prop:prod-obstruction-type1} is sharp in the following sense.

\begin{proposition} \label{prop:prod-existence-type1}
Let $T_1 = \bS^1 \times \bS^1$ and let $\omega_\sym \in \Omega^2(\C P^2)$ be the standard symplectic form. There exists a non-constant map $F\colon \R^4 \to T_1 \times \C P^2$ which is a quasiregular curve with respect to $\omega_{T_1}=\pi_1^* \vol_{T_1} \wedge \pi_2^* \omega_\sym$.
\end{proposition}

Propositions~\ref{prop:prod-obstruction-type1} and \ref{prop:prod-existence-type1} together yield immediately Corollary~\ref{cor:(non-)existence-QRC-type1}.

Before discussing the proofs of Propositions~\ref{prop:prod-obstruction-type2}, \ref{prop:prod-obstruction-type1}, and \ref{prop:prod-existence-type1}, we briefly recall the definition of the standard symplectic form on a complex projective space. For $m\ge 1$, \emph{the standard symplectic form on $\C P^m$} is the unique form $\omega_\sym \in \Omega^2(\C P^m)$ satisfying
\[
U^* \omega_\sym = \omega_\sym
\]
for each unitary $(m+1)\times (m+1)$ matrix $U\in U(m+1)$ and
\[
\iota^* \omega_\sym = \vol_{\C P^1},
\]
where $\iota \colon \C P^1 \to \C P^m$ is the inclusion $[z_0 : z_1] \mapsto [z_0 : z_1 : 0 : \cdots : 0]$ in homogeneous coordinates; see e.g.~\cite[Section 3.3.2]{VO} for an explicit formula of $\omega_\sym$.

\subsection*{Construction of a quasiregular $\omega_{T_1}$-curve $\R^4 \to T_1 \times \C P^2$}

In preparation for the proof of Proposition~\ref{prop:prod-existence-type1}, we show that the natural inclusion $\C P^1 \hookrightarrow \C P^2$ is a quasiregular curve with respect to the standard symplectic form.

\begin{lemma} \label{lem:inclusion}
Let $1\le i\le \ell$ and let $\iota \colon \C P^i \to \C P^\ell$ be the inclusion $[z_0 : \cdots : z_i] \mapsto [z_0 : \cdots : z_i : 0 : \cdots : 0]$. Let $\omega_\sym \in \Omega^2(\C P^\ell)$ be the standard symplectic form on $\C P^\ell$. Then $\iota$ is a $1$-quasiregular curve with respect to $\omega_\sym^{\wedge i}=\omega_\sym \wedge \cdots \wedge \omega_\sym \in \Omega^{2i}(\C P^\ell)$.
\end{lemma}

\begin{proof}
The claim follows immediately since $\iota^* \omega_\sym \in \Omega^2(\C P^i)$ is the standard symplectic form on $\C P^i$.
\end{proof}

We also show that mappings of bounded length distortion $\R^n \to N_i$ induce quasiregular curves $(\R^n)^m \to N_1 \times \cdots \times N_m$; we refer to \cite{MV} and \cite{HR} for more discussion on BLD-maps.

\begin{lemma} \label{lem:product-of-uniform-is-qr}
Let $n\ge 2$ and let $N_1,\ldots,N_m$ be connected and oriented Riemannian manifolds with $\dim N_i \ge n$. Let $\omega_i \in \Omega^n(N_i)$ be closed forms satisfying $\norm{\pi_1^* \omega_1 \wedge \cdots \wedge \pi_m^* \omega_m} \le 1$. Let $L\ge m$ and let $\psi_i \in W_\loc^{1,n}(\R^n,N_i)$ be maps satisfying $\norm{D\psi_i} \le L$ a.e.~in $\R^n$ and $\star \psi_i^* \omega_i \ge L^{-n}$ a.e.~in $\R^n$. Then the map $\psi \colon (\R^n)^m \to N_1 \times \cdots \times N_m$, $(x_1,\ldots,x_m)\mapsto (\psi_1(x_1),\ldots,\psi_m(x_m))$, is a $(L^{4n})^m$-quasiregular curve with respect to $\pi_1^* \omega_1 \wedge \cdots \wedge \pi_m^* \omega_m$.
\end{lemma}

\begin{proof}
For a.e.~$x=(x_1,\ldots,x_m)\in (\R^n)^m$, we have
\begin{align*}
\norm{D\psi(x)}^{nm} &\le \left( \sum_{i=1}^m \norm{D\psi_i(x_i)} \right)^{nm} \le m^{nm} \sum_{i=1}^m \norm{D\psi_i(x_i)}^{nm} \\
&\le m^{nm+1} L^{nm} \le L^{3nm} \le L^{4nm} \prod_{i=1}^m (\star \psi_i^* \omega_i)(x_i) \\
&= L^{4nm} (\star \psi^* (\pi_1^* \omega_1 \wedge \cdots \wedge \pi_m^* \omega_m))(x).
\end{align*}
Since $(\pi_1^* \omega_1 \wedge \cdots \wedge \pi_m^* \omega_m)(\psi(x))\ne 0$ for a.e.~$x\in (\R^n)^m$, the claim follows.
\end{proof}

We are now ready to prove Proposition~\ref{prop:prod-existence-type1}.

\begin{proof}[Proof of Proposition~\ref{prop:prod-existence-type1}]
Let $\exp \colon \R^2 \to T_1$ be the covering map $(x,y)\mapsto (e^{2\pi ix},e^{2\pi iy})$. Let $A\colon \R^2 \to \C P^1$ be an Alexander map (for the construction see e.g.~\cite[Section 6.5.4]{Iw-Ma-book}) and let $L\ge 1$ a constant satisfying $\norm{DA} \le L$ a.e.~in $\R^2$ and $\star A^* \vol_{\C P^1} \ge L^{-2}$ a.e.~in $\R^2$. Let $\iota \colon \C P^1 \to \C P^2$ be the inclusion $[z_0 : z_1]\mapsto [z_0 : z_1 : 0]$. Then, by Lemma~\ref{lem:inclusion},
\[
\norm{D(\iota \circ A)} \le \norm{D\iota \circ A} \norm{DA} \le L
\]
a.e.~in $\R^2$ and
\[
\star(\iota \circ A)^* \omega_\sym = \star A^* \iota^* \omega_\sym = \star A^* \vol_{\C P^1} \ge L^{-2}
\]
a.e.~in $\R^2$. Thus, by Lemma~\ref{lem:product-of-uniform-is-qr}, the map
\[
F\colon \R^4 \to T_1 \times \C P^2, \quad (x,y,z,w)\mapsto (\exp(x,y),\iota(A(z,w))),
\]
is a quasiregular curve with respect to $\pi_1^* \vol_{T_1} \wedge \pi_2^* \omega_\sym$. Since $F$ is non-constant, this concludes the proof.
\end{proof}

\subsection*{Obstruction for quasiregular $\omega_{T_g}$-curves $\R^4 \to T_g \times \C P^2$}

The quasiregular $\omega_{T_1}$-curve $\R^4 \to T_1 \times \C P^2$ constructed in the proof of Proposition~\ref{prop:prod-existence-type1} arose from a pair $(\exp,\iota \circ A)$, where $\exp \colon \R^2 \to T_1$ and $\iota \circ A\colon \R^2 \to \C P^2$ are quasiregular maps. By the classical quasiconformal theory, there does not exist a quasiregular map $\R^2 \to T_g$ if $g\ge 2$, so a priori Proposition~\ref{prop:prod-obstruction-type1}, we know that a quasiregular $\omega_{T_g}$-curve $\R^4 \to T_g \times \C P^2$ cannot be constructed using quasiregular coordinate maps. Proposition~\ref{prop:prod-obstruction-type1} shows that there does not exist any method of constructing a non-constant quasiregular $\omega_{T_g}$-curve $\R^4 \to T_g \times \C P^2$.

The proof of Proposition~\ref{prop:prod-obstruction-type1} is based on combining the following elementary algebraic lemma with Theorem~\ref{thm:main}.

\begin{lemma} \label{lem:algebra}
Let $N$ be a smooth manifold and let $\Phi \colon H_\dR^*(N) \to \bigwedge^* \R^n$ be a graded algebra homomorphism. Let $k,k'>0$. Suppose $c\in H_\dR^k(N)$ and $c'\in H_\dR^{k'}(N)$ are de Rham classes satisfying $\Phi c\wedge \Phi c'\ne 0$. Suppose $c_1,\ldots,c_m \in H_\dR^1(N)$ are de Rham classes for which $c\wedge c_i=0$ for $i=1,\ldots,m$. Suppose $c_1',\ldots,c_m' \in H_\dR^{k-1}(N)$ are de Rham classes satisfying $c_i \wedge c_\ell' = \delta_{i\ell}c$, where $\delta_{i\ell}c$ is the Kronecker delta. Then $m<n$.
\end{lemma}

\begin{proof} For every $(\lambda_1,\ldots,\lambda_m)\in \R^m \setminus \{0\}$, we have
\begin{align*}
\left( \sum_{i=1}^m \lambda_i \Phi c_i \right) \wedge \left( \sum_{\ell=1}^m \lambda_\ell \Phi c_\ell' \right) &= \sum_{i,\ell=1}^m \lambda_i \lambda_\ell \Phi c_i \wedge \Phi c_\ell' =  \sum_{i,\ell=1}^m \lambda_i \lambda_\ell \Phi(c_i \wedge c_\ell') \\
&= \sum_{i=1}^m \lambda_i^2 \Phi c = \left( \sum_{i=1}^m \lambda_i^2 \right) \Phi c \ne 0,
\end{align*}
so $\sum_{i=1}^m \lambda_i \Phi c_i \ne 0$. We conclude that $\Phi c_1,\ldots,\Phi c_m \in \bigwedge^1 \R^n$ are linearly independent, which yields that $m\le \dim \bigwedge^1 \R^n = n$.

Suppose now that $m=n$. Then $\Phi c_1,\ldots,\Phi c_m$ is a basis of $\bigwedge^1 \R^n$ and hence $\{\Phi c_{i_1} \wedge \cdots \wedge \Phi c_{i_{k'}}\}_I$ is a basis of $\bigwedge^{k'} \R^n$. Since $\Phi c'\in \bigwedge^{k'} \R^n$, there exist coefficients $\mu_I \in \R$ satisfying
\[
\Phi c' = \sum_I \mu_I \Phi c_{i_1} \wedge \cdots \wedge \Phi c_{i_{k'}}.
\]
Hence, we arrive at a contradiction
\begin{align*}
\Phi c\wedge \Phi c' &= \Phi c \wedge \left( \sum_I \mu_I \Phi c_{i_1} \wedge \cdots \wedge \Phi c_{i_{k'}} \right) = \sum_I \mu_I \Phi c \wedge \Phi c_{i_1} \wedge \cdots \wedge \Phi c_{i_{k'}} \\
&= \sum_I \mu_I \Phi (c \wedge c_{i_1} \wedge \cdots \wedge c_{i_{k'}}) = 0.
\end{align*}
\end{proof}

We are now ready to prove Proposition~\ref{prop:prod-obstruction-type1}.

\begin{proof}[Proof of Proposition~\ref{prop:prod-obstruction-type1}]
Suppose that $F$ is non-constant. By Theorem~\ref{thm:signed}, we have $\int_{\R^n} F^* \omega_{T_g} = \infty$. Thus, by Theorem~\ref{thm:main}, there exists a graded algebra homomorphism $\Phi \colon H_\dR^*(T_g \times \C P^2) \to \bigwedge^* \R^4$ for which $\Phi [\omega_{T_g}]_\dR \ne 0$.

Let $c_1,\ldots,c_{2g}$ be basis of $H_\dR^1(T_g)$ satisfying $c_i \wedge c_{i+1} = [\vol_{T_g}]_\dR$ for $i=1,3,\ldots,2g-1$ and $c_i \wedge c_\ell = 0$ for every pair $(i,\ell)\in \{1,\ldots,2g\}^2$ with $|2\ell-2i-1|>1$. Denote $c_i' = c_{i+1}$ for $i=1,3,\ldots,2g-1$ and $c_i' = -c_{i-1}$ for $i=2,4,\ldots,2g$. Since
\[
\Phi \pi_1^* [\vol_{T_g}]_\dR \wedge \Phi \pi_2^* [\omega_\sym]_\dR = \Phi [\omega_{T_g}]_\dR \ne 0,
\]
\[
\pi_1^* [\vol_{T_g}]_\dR \wedge \pi_1^* c_i = \pi_1^* ([\vol_{T_g}]_\dR \wedge c_i) = 0,
\]
and
\[
\pi_1^* c_i \wedge \pi_1^* c_\ell' = \pi_1^* (c_i \wedge c_\ell') = \pi_1^* (\delta_{i\ell} [\vol_{T_g}]_\dR) = \delta_{i\ell} \pi_1^* [\vol_{T_g}]_\dR,
\]
Lemma~\ref{lem:algebra} yields that $2g<4$. This is a contradiction since $g\ge 2$.
\end{proof}

\subsection*{Obstruction for quasiregular $\omega_{S_\nu}$-curves $\R^4 \to S_\nu \times \C P^2$}

Proposition~\ref{prop:prod-obstruction-type2} follows from Theorem~\ref{thm:main} and the following elementary observation. Since the proof is similar to the proof of Lemma~\ref{lem:algebra}, we omit the details.

\begin{lemma} \label{lem:algebra-2}
Let $N$ be a smooth manifold and let $\Phi \colon H_\dR^*(N) \to \bigwedge^* \R^n$ be a graded algebra homomorphism. Let $k\ge 2$ and suppose $c\in H_\dR^k(N)$ is a de Rham class satisfying $\Phi c\ne 0$. Let $1\le k'\le k-1$. Suppose $c_1,\ldots,c_m \in H_\dR^{k'}(N)$ and $c_1',\ldots,c_m' \in H_\dR^{k-k'}(N)$ are de Rham classes for which $c_i \wedge c_\ell' = \delta_{i\ell}c$, where $\delta_{i\ell}c$ is the Kronecker delta. Then $m\le \binom{n}{k'}$.
\end{lemma}


\begin{proof}[Proof of Proposition~\ref{prop:prod-obstruction-type2}]
Suppose that $F$ is non-constant. By Theorem~\ref{thm:signed}, we have $\int_{\R^n} F^* \omega_{S_\nu} = \infty$. Thus, by Theorem~\ref{thm:main}, there exists a graded algebra homomorphism $\Phi \colon H_\dR^*(S_\nu \times \C P^2) \to \bigwedge^* \R^6$ for which $\Phi [\omega_{S_\nu}]_\dR \ne 0$.

Let $c_1,\ldots,c_{2\nu}$ be basis of $H_\dR^2(S_\nu)$ satisfying $c_i \wedge c_{i+1} = [\vol_{S_\nu}]_\dR$ for $i=1,3,\ldots,2\nu-1$ and $c_i \wedge c_\ell = 0$ for every pair $(i,\ell)\in \{1,\ldots,2\nu\}^2$ with $|2\ell-2i-1|>1$. Denote $c_i' = c_{i+1}$ for $i=1,3,\ldots,2\nu-1$ and $c_i' = c_{i-1}$ for $i=2,4,\ldots,2\nu$. Since $\Phi \pi_1^* [\vol_{S_\nu}]_\dR \ne 0$ and
\[
\pi_1^* c_i \wedge \pi_1^* c_\ell' = \pi_1^* (c_i \wedge c_\ell') = \pi_1^* (\delta_{i\ell} [\vol_{S_\nu}]_\dR) = \delta_{i\ell} \pi_1^* [\vol_{S_\nu}]_\dR,
\]
Lemma~\ref{lem:algebra-2} yields that $2\nu \le \binom{6}{2} = 15$. This is a contradiction since $\nu \ge 8$.
\end{proof}


\bibliographystyle{abbrv}

\end{document}